\documentclass[reqno,11pt]{amsart}
\usepackage[left=3cm,right=3cm,top=3cm,bottom=2.5cm, marginparwidth=2cm]{geometry}
\usepackage{graphicx} 
\usepackage[utf8]{inputenc}
\usepackage[T1]{fontenc}
\usepackage{amsmath,amssymb,mathrsfs,amsfonts}
\usepackage{amsthm}
\usepackage[all]{xy}
\usepackage[hypertexnames=false,backref=page,pdftex,
 	pdfpagemode=UseNone,
 	breaklinks=true,
 	extension=pdf,
 	colorlinks=true,
 	linkcolor=blue,
 	citecolor=red,
 	urlcolor=blue,
 ]{hyperref}

\newcounter{ct}
\counterwithin{ct}{section}
\newtheorem{lem}[ct]{Lemma}
\newtheorem{prop}[ct]{Proposition}
\newtheorem{thm}[ct]{Theorem}

\newtheorem{conj}[ct]{Conjecture}
\theoremstyle{definition}
\newtheorem{df}[ct]{Definition}
\theoremstyle{remark}
\newtheorem{rmk}[ct]{Remark}

\usepackage{nicematrix}

\newcommand{\R}{\mathbb{R}}
\newcommand{\Q}{\mathbb{Q}}
\newcommand{\Z}{\mathbb{Z}}

\newcommand{\Id}{\mathrm{Id}}

\newcommand{\HK}{hyperk\"ahler }

\newcommand{\Kah}{\mathrm{K\ddot{a}h}}
\newcommand{\Nef}{\mathrm{Nef}}

\newcommand{\Aut}{\mathrm{Aut}}

\title[transcendental BPF conjecture for Calabi-Yau manifolds]{A note on the transcendental basepoint-free conjecture for Calabi-Yau manifolds}
\author[B. Philippe]{Bastien Philippe}
\address{Université de Lorraine, CNRS, IECL, F-54000 Nancy, France}
\email{bastien.philippe@univ-lorraine.fr}

\begin{document}

\maketitle

\begin{abstract}
    In this note, we prove that the transcendental basepoint-free conjecture for Calabi-Yau manifolds holds if it holds for its \HK factors in its Beauville-Bogomolov decomposition. Based on a contraction theorem due to Bakker and Lehn, we show that the conjecture holds for a big and nef class $\alpha$ on a \HK manifold under a mild condition on the dimension of the space generated by classes of rational curves on which $\alpha$ vanishes. 
\end{abstract}

\section{Introduction}

Let $X$ be a compact K\"ahler manifold. A class $\alpha \in H^{1,1}(X, \R)$ is said to be \textit{semi-ample} if there exists a holomorphic contraction $\rho : X \to X'$ where $X'$ is a normal analytic space together with a K\"ahler class $\omega \in H^{1,1}_{\mathrm{BC}}(X')$ such that $\rho^*\omega = \alpha$. The goal of this note is to study the following conjecture due to Tosatti.

\begin{conj}[{\cite[Conjecture 6.1]{TosattiLimitsOfCY}, \cite[Conjecture 2.7]{tosatti2025ricciflatmetricscalabiyaumanifolds}}]\label{conj Tosatti}
    Let $X$ be a Calabi-Yau manifold and let $\alpha \in H^{1,1}(X, \R)$ be a big and nef class. Then $\alpha$ is semi-ample.
\end{conj}

Here, the term \textit{Calabi-Yau} manifold refers to a compact K\"ahler manifold $X$ such that $c_1(K_X) = 0 \in H^{1,1}(X, \R)$. The above conjecture can be seen as a special case of a conjectural generalization to the transcendental setting of Kawamata's basepoint-free theorem and can also be formulated for non-Calabi-Yau manifolds, see for example \cite[Conjecture 1.2]{FilipTosatti}. In the case where $X$ is projective and $\alpha$ is the class of a big and nef $\R$-divisor, Conjecture \ref{conj Tosatti} follows from the usual basepoint-free theorem (\cite[Theorem 3.9.1]{BCHM}). Conjecture \ref{conj Tosatti} is also known when $\dim(X) \leqslant 3$ (\cite{FilipTosatti, HoringAdjoint}) and when $X$ is projective (\cite{DasHaconTranscendentalMMPforProj}).

Our first result is that Conjecture \ref{conj Tosatti} can be reduced to the case of \HK manifolds. Let $X$ be a Calabi-Yau manifold. The Beauville-Bogomolov decomposition theorem (\cite{BeauvillePremièreClasse}) asserts that there exists a finite étale map $\pi : \widehat{X} \to X$ such that $\widehat{X}$ decomposes as a product of the form 
\[ \widehat{X} = T \times \prod_{i = 1}^{n} X_i \times \prod_{j = 1}^{m} Y_j \]
where $T$ is a torus, the $X_i$'s are \HK manifolds and the $Y_i$'s are strict Calabi-Yau manifolds. In the setting of Conjecture \ref{conj Tosatti}, since every big and nef class on a torus is K\"ahler we can assume that there is at least a \HK or strict Calabi-Yau factor in the Beauville-Bogomolov decomposition. In this case, the class $\pi^*\alpha$ induces big and nef classes on each factor of the decomposition of $\widehat{X}$ (Lemma \ref{lem decompo classes}).

\begin{prop}\label{prop reduction CY to HK}
    Let $X$ be a Calabi-Yau manifold and let $\alpha \in H^{1,1}(X, \R)$ be a big and nef class. Let 
    \[\pi : T \times \prod_{i = 1}^{n} X_i \times \prod_{j = 1}^{m} Y_j  \to X\]
    be a Beauville-Bogomolov decomposition of $X$, where the $X_i$'s are the \HK factors. Let $\alpha_i \in H^{1,1}_{\mathrm{BC}}(X_i)$ be the big and nef class induced by $\pi^*\alpha$ on $X_i$ (see Lemma \ref{lem decompo classes}). Assume that for all $1 \leqslant i \leqslant n$, $\alpha_i$ is semi-ample. Then $\alpha$ is semi-ample.
\end{prop}

In the \HK case, we use a contraction theorem due to Bakker and Lehn (\cite{BakkerLehnAGlobalTorelli}) to establish the following partial result.

\begin{prop}\label{prop cas HK}
    Let $X$ be a \HK manifold and let $\alpha \in H^{1,1}(X, \R)$ be a big and nef class. Assume that
    \[ N := \mathrm{span}_\Q\{[C] \in H_2(X_i, \Z) \mid \text{$C \subset X_i$ is a rational curve and $[C] \cdot \alpha = 0$} \} \subset H_2(X_i, \Q)\]
    is such that $\dim(N) \leqslant b_2(X) - 5$. Then $\alpha$ is semi-ample. 
\end{prop}

Let us now discuss the proofs of these results. The proof of Proposition \ref{prop reduction CY to HK} relies on the following observations : \begin{itemize}
    \item for tori, Conjecture \ref{conj Tosatti} is automatically satisfied as every big and nef class is K\"ahler ;
    \item Calabi-Yau manfiolds carry no global holomorphic 2-forms, thus they are always projective and every $(1,1)$ class is the class of a $\R$-divisor. Conjecture \ref{conj Tosatti} thus follows from the basepoint-free theorem for $\R$-divisors on projective manifolds (\cite[Theorem 3.9.1]{BCHM}).
\end{itemize}
In particular, Proposition \ref{prop reduction CY to HK} amounts to showing that Conjecture \ref{conj Tosatti} is stable under product and finite étale morphism. This is to content of Section \ref{section proof reduction}. Let us point out that in \cite{HoringAdjoint} H\"oring has treated the case of K-trivial threefolds with at most terminal singularities, also by using the Beauville-Bogomolov decomposition. In his proof, he reduces to the case of a product of an elliptic curve with a K3 surface and obtains an explicit description as a product of the null locus associated to the corresponding big and nef class to show that the contraction descends under the quasi-étale map. Our proof differs from his in this aspect.

To prove Proposition \ref{prop cas HK}, we first observe that the assumption implies that $\alpha$ lies inside the interior of the positive cone with respect to the Beauville-Bogomolov form. We then use the description of the K\"ahler cone obtained by Huybrechts (\cite{HuybrechtsKahlerCone}) and Boucksom (\cite{BoucksomConeHK}) together with the boundedness of squares of MBM classes (\cite{AmerikVerbitskyannallesENS}) to conclude that $\alpha$ belongs to a face of the nef cone of $X$ defined as the orthogonal of a set of rational curves on $X$. We then use a contraction theorem of Bakker and Lehn (\cite{BakkerLehnAGlobalTorelli}) stating that, provided that the dimension of this face isn't too low, there exists a contraction morphism contracting precisely the curves defining the face. Finally, using techniques from the transcendental minimal model program (especially Lemma \ref{lem classes Kahler} and Lemma \ref{lem pushforward}), we show that there exists a K\"ahler class on the contracted variety pulling back to $\alpha$. This last part addresses a point made in \cite{BakkerLehnAGlobalTorelli} concerning the refined period mapping attached to the locally trivial marked moduli space of primitive symplectic varieties admitting crepant resolutions, see Remark \ref{rmk module}. 

We end this introduction with a few comments on the assumption in Proposition \ref{prop cas HK}. First, the maximal possible dimension for $N$ is $b_2(X) - 3$. This bound is however attained ; there are examples of contractions $\rho : X \to X'$ where $X$ is a \HK manifold and $X'$ is K\"ahler with $b_2(X') = 3$ (\cite[Remark 5.10]{BakkerLehnAGlobalTorelli}). The assumption is needed in order to apply \cite{AmerikVerbitskyannallesENS} and \cite{BakkerLehnAGlobalTorelli}, both of which ultimately relying on the deep work of Verbitsky on the orbits of the action of monodromy groups on moduli spaces of marked \HK manifolds and associated period domains (\cite{VerbitskyErgodic, VerbitskyErgodicErratum}). The contraction used in Proposition \ref{prop cas HK} is constructed in \cite{BakkerLehnAGlobalTorelli} as a small deformation of a contraction on a projective manifold, where the last one is produced by applying the usual basepoint-free theorem. It would be interesting to see if this approach consisting of obtaining the desired contraction as a deformation of one on a projective manifold could be carried further to fully show Conjecture \ref{conj Tosatti}. Obtaining the analogue of the surjectivity of the period map in \cite[Theorem 8.1]{HuybrechtsBasicResults} for the locally trivial marked moduli space of primitive symplectic varieties admitting crepant resolutions with arbitrary second Betti number (which is proved for $b_2 \geqslant 5$ in \cite{BakkerLehnAGlobalTorelli}) could be a key step in this approach. 

\subsection*{Acknowledgements}
The author wishes to express his gratitude to his advisors Beno\^it Cadorel and Matei Toma for their guidance and encouragements to write this note. 

\section{Preliminaries}

\subsection{Positivity for K\"ahler spaces}

Let $X$ be a normal compact complex space. We will consider forms and currents on such $X$, for this we refer to \cite{DemaillyMesureDeMA}. We denote by $\mathrm{PH}_X$ the sheaf of pluri-harmonic functions on $X$. This sheaf coincides with the sheaf of real parts of holomorphic functions on $X$ (\cite[Lemma 4.6.1]{BoucksomGuedjRegularizing}).

\begin{df}[{\cite[Definition 3.1]{HoringPeternellMMP3fold}, see also \cite[Definition 4.6.2]{BoucksomGuedjRegularizing}}]
    We define the \textit{Bott-Chern cohomology} by $H^{1,1}_{\mathrm{BC}}(X) := H^1(X, \mathrm{PH}_X)$.
\end{df}

The short exact sequence associated to the quotient sheaf $\mathscr{C}^\infty_X/\mathrm{PH}_X$ gives rise to a natural map $H^0(X, \mathscr{C}^\infty_X/\mathrm{PH}_X) \to H^{1,1}_{\mathrm{BC}}(X)$ which is surjective (\cite{BoucksomGuedjRegularizing}). Moreover, by considering the short exact sequence 
\[ 0 \to \underline{\R} \xrightarrow{i \cdot} \mathcal{O}_X \xrightarrow{\mathrm{Re}} \mathrm{PH}_X \to 0 \]
we get a map $H^{1,1}_{\mathrm{BC}}(X) \to H^2(X, \R)$. If $X$ has rational singularities, then this map is injective (\cite{BoucksomGuedjRegularizing}) and its image coincides with $H^{1,1}(X, \R) := H^{1,1}(X) \cap H^2(X, \R)$ where $H^{1,1}(X)$ is the $(1,1)$ part of the pure Hodge structure of weight 2 on $H^2(X, \Z)^{\mathrm{tf}}$ (\cite[Lemma 2.1]{BakkerLehnAGlobalTorelli}). In this case, we can define the intersection product on $H^{1,1}_{\mathrm{BC}}(X)$ via the cup-product on $H^2(X, \R)$ (\cite[Remark 3.7]{HoringPeternellMMP3fold}). 

\begin{df}
    A \textit{K\"ahler metric} is by definition an element of $H^0(X, \mathscr{C}^\infty_X/\mathrm{PH}_X)$ that can be represented by a family $(U_i, \varphi_i)_{i \in I}$ such that for all $i \in I$, $\varphi_i$ is a smooth strictly plurisubharmonic function defined on $U_i$. The space $X$ will be called \textit{K\"ahler} if it supports a K\"ahler metric and a class of $H^{1,1}_{\mathrm{BC}}(X)$ will be called \textit{K\"ahler} if it is the image of a K\"ahler metric under the natural map $H^0(X, \mathscr{C}^\infty_X/\mathrm{PH}_X) \to H^{1,1}_{\mathrm{BC}}(X)$.
\end{df}

The set of K\"ahler classes inside $H^{1,1}_{\mathrm{BC}}(X)$ will still be denoted by $\Kah(X)$ and is, as in the smooth case, an open convex cone inside $H^{1,1}_{\mathrm{BC}}(X)$ (\cite[Corollary 3.7]{GrafKirschnerFiniteQuotients}). When $X$ is a K\"ahler space, we will also use weaker positivity notions for classes in $H^{1,1}_{\mathrm{BC}}(X)$. A class in $H^{1,1}_{\mathrm{BC}}(X)$ is said to be \textit{nef} if it belongs to $\Nef(X) := \overline{\Kah(X)}$. A nef class $\alpha \in H^{1,1}_{\mathrm{BC}}(X)$ will be called \textit{big} if $\alpha^{\dim(X)} > 0$ (see \cite[Lemma 2.35]{DasHaconPaun4-dimensional} for the equivalence with the usual notion).

\begin{df}[{\cite[Definition 3.8]{HoringPeternellMMP3fold}}]
    Let $X$ be a normal compact complex space. We define $N_1(X)$ to be the space of real currents of bidimension $(1,1)$ modulo the equivalence relation given by $T \sim T'$ if and only if for $T(\eta) = T'(\eta)$ for all form all real closed $(1,1)$ form $\eta$ with local potentials. We denote by $\overline{\mathrm{NA}}(X)$ the closed cone generated by positive currents inside $N_1(X)$. 
\end{df}

Here $\overline{\mathrm{NA}}(X)$ play the role of the cone of curves for complex analytic spaces that are not necessarily projective. If $X$ has at worst rational singularities we have an isomorphism $\Psi : H^{1,1}_{\mathrm{BC}}(X) \to N_1(X)^*$ (\cite[Proposition 3.9]{HoringPeternellMMP3fold}). 

\begin{lem}[{\cite[Lemma 2.3]{DasHaconTranscendentalMMPforProj}}]\label{lem classes Kahler}
    Let $X$ be a normal compact K\"ahler sapce with at most rational singularities. Then $\Nef(X)$ and $\overline{\mathrm{NA}}(X)$ are dual to each other via $\Psi$. In particular, a class $\alpha \in H^{1,1}_{\mathrm{BC}}(X)$ is K\"ahler if and only $\alpha \cdot T > 0$ for all $T \in \overline{\mathrm{NA}}(X) \setminus \{0\}$. 
\end{lem}

\subsection{Contractions}

\begin{df}
    A \textit{contraction} is a proper and surjective holomorphic map $\rho : X \to X'$ such that $X$ and $X'$ are normal and $\rho$ has connected fibers. 
\end{df}

In the case where $X$ and $X'$ both have at worst rational singularities, cohomology is well-behaved under contraction.

\begin{prop}[{\cite[Lemma 3.3]{HoringPeternellMMP3fold}, \cite[Lemma 2.1]{BakkerLehnAGlobalTorelli}}]\label{prop HP pullback}
    Let $\rho : X \to X'$ be a contraction. Assume that $X$ and $X'$ have at most rational singularities. Then the morphism $\rho^* : H^{1,1}_{\mathrm{BC}}(X') \to H^{1,1}_{\mathrm{BC}}(X)$ is injective and its image is exactly the set of classes inside $H^{1,1}_{\mathrm{BC}}(X)$ that vanish when intersected with curves that are contracted by $\rho$.
\end{prop}

We will also use the following. 

\begin{lem}[{\cite[Lemma 2.1]{DasHaconTranscendentalMMPforProj}}]\label{lem pushforward}
    Let $X$ and $X'$ be normal K\"ahler spaces with at most rational singularities and $\rho : X \to X'$ be a contraction. Then $\overline{\mathrm{NA}}(X') = \rho_*\overline{\mathrm{NA}}(X)$. 
\end{lem}

Let us now determine the behavior of the K\"ahler cone under contraction.

\begin{lem}\label{lem cone kahler contracté}
    Let $\rho : X \to X'$ be a contraction. Assume that both $X$ and $X'$ are K\"ahler varieties with at worst rational singularities. Let 
    \[N := \mathrm{span}_\R\{T_C \in N_1(X) \mid \text{$C$ is a curve contracted by $\rho$}\} \subset N_1(X) \]
    and let $\tau := N^\perp \subset H^{1,1}_{\mathrm{BC}}(X)$. Then $\rho^*(\Kah(X')) = \mathrm{Int}(\Nef(X) \cap \tau)$.
\end{lem}

Here $T_C$ denotes the current of integration over the curve $C$. Recall that for a closed convex cone $\mathscr{C}$ inside a finite dimensional vector space $V$, we call a linear subspace $\tau \subset V$ an \textit{extremal face} if $\tau \cap \mathscr{C}$ has non-empty interior inside $\tau$ and if for two elements $x,y \in \mathscr{C}$ such that $x + y \in \tau$ one has $x,y \in \tau$. An extremal face of codimension 1 will be called a \textit{wall}.

\begin{lem}\label{lem cone}
    Let $V$ and $\mathscr{C} \subset V$ be as above and let $\tau$ be an extremal face of $\mathscr{C}$. Let $\varphi : V \to \R$ be a linear function such that for all $x \in \mathscr{C}$ one has $\varphi(x) \geqslant 0$ and let $H:= \ker(\varphi)$ be the associated hyperplane. Assume that there exists $x \in \mathrm{Int}(\tau \cap \mathscr{C})$ such that $x \in H$. Then $\tau \subset H$. 
\end{lem}

\begin{proof}
    Let $y \in \tau \cap \mathscr{C}$. Then since $x \in \mathrm{Int}(\tau \cap \mathscr{C})$ there exists $\varepsilon > 0$ such that $x \pm \varepsilon y \in \tau \cap \mathscr{C}$. By hypothesis $\varphi(x + \varepsilon y) = \varepsilon \varphi(y) \geqslant 0$ and $\varphi(x - \varepsilon y) = -\varepsilon \varphi(y) \geqslant 0$ and thus $\varphi(y) = 0$. In particular $H$ contains $\tau \cap \mathscr{C}$ which has non-empty interior inside $\tau$, and thus contains $\tau$. 
\end{proof}
    
\begin{proof}[Proof of Lemma \ref{lem cone kahler contracté}]
    Let us start by noticing that by Proposition \ref{prop HP pullback}, $\tau$ is the image of $H^{1,1}_{\mathrm{BC}}(X')$ inside $H^{1,1}_{\mathrm{BC}}(X)$. Let $\omega \in H^{1,1}_{\mathrm{BC}}(X')$ be a K\"ahler class. By Lemma \ref{lem classes Kahler} combined by the projection formula, $\rho^*\omega$ is nef. Since $\Kah(X')$ is open inside $H^{1,1}_{\mathrm{BC}}(X')$, we have $\rho^*(\Kah(X')) \subset \mathrm{Int}(\Nef(X) \cap \tau)$. Notice that this implies in particular that $\tau$ is a face of $\Nef(X)$. Let $\alpha \in \mathrm{Int}(\Nef(X) \cap \tau)$. Then again by Proposition \ref{prop HP pullback} there exists $\omega \in H^{1,1}_{\mathrm{BC}}(X')$ such that $\alpha = \rho^*\omega$. By Lemma \ref{lem classes Kahler} it suffices to check that $T' \cdot \omega > 0$ for all $T' \in \overline{\mathrm{NA}}(X') \setminus \{0\}$. Let $T' \in \overline{\mathrm{NA}}(X')$. By Lemma \ref{lem pushforward}, there exists $T \in \overline{\mathrm{NA}}(X)$ such that $T' = \rho_*(T)$. We have $\alpha \in T^\perp$ and thus by Lemma \ref{lem cone} $\tau \subset T^\perp$. Taking the orthogonal we get that $T \in \tau^\perp$ and thus $T$ is represented by a linear combination with real coefficients of currents of integration over curves contracted by $\rho$. In particular $\rho_*T = 0$ and thus $\omega$ is K\"ahler. 
\end{proof}

\section{Proofs}

\subsection{Proof of Proposition \ref{prop reduction CY to HK}}\label{section proof reduction}

Recall that for all product $X \times Y$ of K\"ahler manifolds such that $H^1(X, \mathcal{O}_X) = 0$, we have by the K\"unneth formula for Dolbeault cohomology (\cite[p.105]{GriffithsHarrisPrinciples}) that $H^{1,1}(X \times Y, \R) = \mathrm{pr}^*_X H^{1,1}(X, \R) \oplus \mathrm{pr}^*_Y H^{1,1}(Y, \R)$. In this case, if $\alpha \in H^{1,1}(X \times Y, \R)$, we denote by $\alpha_X \in H^{1,1}(X, \R)$ and $\alpha_Y \in H^{1,1}(Y, \R)$ the classes such that $\alpha = \mathrm{pr}^*_X \alpha_X + \mathrm{pr}^*_Y \alpha_Y$.

\begin{lem}\label{lem decompo classes}
    Let $X$ and $Y$ be K\"ahler manifolds. Assume that $h^1(X, \mathcal{O}_X) = 0$. Then a class $\alpha \in H^{1,1}(X \times Y, \R)$ is big and nef if and only if $\alpha_X$ and $\alpha_Y$ are big and nef classes on $X$ and $Y$ respectively. 
\end{lem}

\begin{proof}
    If $\omega_X$ and $\omega_Y$ are K\"ahler classes on $X$ and $Y$ respectively then $\mathrm{pr}^*_X \omega_X + \mathrm{pr}^*_Y \omega_Y$ is K\"ahler on $X \times Y$. Conversely, if $\omega = \mathrm{pr}^*_X \omega_X + \mathrm{pr}^*_Y \omega_Y$ is K\"ahler on $X \times Y$, then by restricting to submanifolds of the form $X \times \{p\}$ and $\{q\} \times Y$ one obtains that $\omega_X$ and $\omega_Y$ are K\"ahler. Since nef classes are limits of K\"ahler classes, this implies that $\Nef(X \times Y) = \mathrm{pr}^*_X \Nef(X) \oplus \mathrm{pr}^*_Y \Nef(Y)$. For bigness, let us observe that for any class $\alpha = \mathrm{pr}^*_X \alpha_X + \mathrm{pr}^*_Y \alpha_Y$ one has 
    \[ \int_{X \times Y} \alpha^{\dim(X) + \dim(Y)} = \left(\int_X \alpha_X^{\dim(X)} \right) \cdot \left(\int_Y \alpha_Y^{\dim(Y)} \right) \]
    and in particular, if $\alpha$ is nef, then $\alpha$ is big if and only if $\alpha_X$ and $\alpha_Y$ are big. 
\end{proof}

We will need to make contractions descend under finite étale maps. For this, we first treat the Galois case. More generally, we establish the following. We use the term \textit{cover} to designate a finite surjective morphism. A cover $\pi : \widehat{X} \to X$ is said to be \textit{Galois} of group $G$ if $G$ is a finite subgroup of $\Aut(\widehat{X})$ such that $X = \widehat{X}/G$ and $\pi$ is the projection map. 

\begin{lem}\label{lem Galois}
    Let $\pi : \widehat{X} \to X$ be a Galois cover of group $G$ between compact K\"ahler manifolds. Let $\alpha \in H^{1,1}(X, \R)$ and $\widehat{\alpha} := \pi^*\alpha$. Assume that there exists a bimeromorphic contraction $\widehat{\rho} : \widehat{X} \to \widehat{Y}$, with $\widehat{Y}$ having at most rational singularities, together with a K\"ahler class $\widehat{\omega} \in H^{1,1}_{\mathrm{BC}}(\widehat{Y})$ such that $\widehat{\rho}^*\widehat{\omega} = \widehat{\alpha}$. Then there exists a bimeromorphic contraction $\rho : X \to Y$ together with a K\"ahler class $\omega \in H^{1,1}_{\mathrm{BC}}(Y)$ such that $\rho^*\omega = \alpha$, i.e. $\alpha$ is semi-ample.
\end{lem}

\begin{proof}
    First, we show that the action of $G$ on $\widehat{X}$ descends to an action on $\widehat{Y}$. Let $g \in G$. The fibers of $\widehat{\rho}$ are exactly the subvarieties $Z$ of $\widehat{X}$ for which $\alpha|_{Z} \equiv 0$. By construction, one has $g^*\widehat{\alpha} = \widehat{\alpha}$, and thus $g$ must exchange fibers of $\widehat{\rho}$. Since $\widehat{\rho}$ has connected fibers, $\widehat{\rho}$ must be its own Stein factorization. By the universal property of the Stein factorization (\cite[p. 214]{GrauertRemmertCohAnSheaves}), we get that $g$ descends to a map $g' : \widehat{Y} \to \widehat{Y}$. This shows that there is an action of $G$ on $\widehat{Y}$ for which $\widehat{\rho}$ is $G$-equivariant. Let $Y := \widehat{Y}/G$ which is a normal complex space by \cite{CartanQuotientEspace}. By construction, we have the following diagram
    \[ \xymatrix{
    \widehat{X} \ar[d]_{\pi} \ar[rr]^{\widehat{\rho}} && \widehat{Y} \ar[d]^{\pi_Y} \\
    X \ar[rr]^\rho && Y.}
    \]
    Clearly, $\rho$ is surjective. Moreover, any fiber of $\rho$ is isomorphic to a quotient of a fiber of $\widehat{\rho}$ and is in particular connected. 

    It remains to show that $\widehat{\omega}$ descends to a K\"ahler class $\omega \in H^{1,1}_{\mathrm{BC}}(Y)$. For this we follow the arguments given in \cite[Section II.3]{VarouchasMathAnn} and \cite[Proposition 2.1]{VarouchasInvent}. The morphism $\pi_Y$ is geometrically flat in the sense \cite[p.20]{VarouchasMathAnn}, i.e. it induces a morphism $H : Y \to \mathrm{Sym}^d(\widehat{Y})$, where $d$ is the degree of $\pi_Y$. For a function $\varphi$ defined on an open $U$ of $Y$ we consider the function $\widetilde{\varphi}$ defined on $\mathrm{Sym}^d(U)$ by $\widetilde{\varphi}(\sum_{i = 1}^{d} x_i) := \frac{1}{d} \sum_{i = 1}^{d} \varphi(x_i)$. Let $k \in H^0(\widehat{Y}, \mathscr{C}^\infty_{\widehat{Y}}/\mathrm{PH}_{\widehat{Y}})$ be a K\"ahler metric in the class $\widehat{\omega}$. Let $(U_i, \varphi_i)_{i \in I}$ be a representative of $k$ where $(U_i)_{i \in I}$ is an open covering of $\widehat{Y}$ and $(\varphi_i)_{i \in I}$ is a family of smooth strictly plurisubharmonic functions defined on each $U_i$. Since $\pi_Y$ is finite and surjective, up to restricting the $U_i$'s and forming finite disjoint union, we can assume that there is an open covering $(V_i)_{i \in I}$ of $Y$ such that for all $i \in I$, $\pi_Y^{-1}(V_i) = U_i$. We can then consider for all $i \in I$ the function $\psi_i := \widetilde{\varphi} \circ H$ which are continuous strictly plurisubharmonic functions (\cite[Lemma 3.1.2]{VarouchasMathAnn}) and define a section $k' \in H^0(Y, \mathscr{C}^0_{Y}/\mathrm{PH}_{Y})$ (\cite[p.27]{VarouchasMathAnn}). Notice that by construction we have $\pi_Y^*k' = \frac{1}{|G|} \sum_{g \in G} g^*k$. As for $\mathscr{C}^\infty$ functions, we have a map $H^0(Y, \mathscr{C}^0_{Y}/\mathrm{PH}_{Y}) \to H^{1,1}_{\mathrm{BC}}(Y)$. We denote by $\omega \in H^{1,1}_{\mathrm{BC}}(Y)$ the image $k'$ under this map. By Proposition \ref{prop HP pullback}, $\widehat{\rho}^*$ embeds $H^{1,1}_{\mathrm{BC}}(\widehat{Y})$ into $H^{1,1}_{\mathrm{BC}}(\widehat{X})$. In particular, we have $g^*\widehat{\omega} = \widehat{\omega}$ for all $g \in G$ and thus $\pi_Y^*\omega = \widehat{\omega}$. By \cite[Theorem 1]{VarouchasMathAnn} we can find a K\"ahler metric $k'' \in H^0(Y, \mathscr{C}^{\infty}_{Y}/\mathrm{PH}_{Y})$ whose class in $H^{1,1}_{\mathrm{BC}}(Y)$ is $\omega$. By construction, we have $(\rho \circ \pi)^*\omega = \widehat{\alpha}$ and since $\pi^* : H^{1,1}(X, \R) \to H^{1,1}(\widehat{X}, \R)$ is injective, we must have $\rho^*\omega = \alpha$.
\end{proof}

\begin{proof}[Proof of Proposition \ref{prop reduction CY to HK}]
    Recall that for any étale cover such as $\pi$ there exists a Galois closure, that is a finite étale cover $\pi' : \widehat{X}' \to \widehat{X}$ such that $\pi \circ \pi'$ is Galois. Since $\widehat{X}$ is a product of a torus with a simply connected manifold, the covering $\pi'$ must come from a covering of the torus factor. In particular, $\widehat{X}'$ is still a product of a torus with hyperk\"ahler and strict Calabi-Yau manifolds. We can thus assume that $\pi$ is Galois from the start. We can also assume that $\alpha$ is not K\"ahler and thus $\widehat{\alpha} := \pi^*\alpha$ is not K\"ahler either. Clearly, $\widehat{\alpha}$ is big and nef. Since every big and nef class on a torus is K\"ahler, there is at least a \HK or strict Calabi-Yau factor in $\widehat{X}$. We can thus apply Lemma \ref{lem decompo classes} to obtain a decomposition
    $\widehat{\alpha} = \mathrm{pr}_T^*\alpha_T + \sum_{i = 1}^{n} \mathrm{pr}_{X_i}^*\alpha_i + \sum_{j = 1}^{m} \mathrm{pr}_{Y_j}^*\beta_j$ where $\alpha_i \in H^{1,1}(X_i, \R)$ and $\beta_j \in H^{1,1}(Y_j, \R)$ are big and nef classes and $\alpha_T \in H^{1,1}(T, \R)$ is K\"ahler. By assumption, for all $1 \leqslant i \leqslant n$ we have a contraction $\rho_i : X_i \to X'_i$ together with K\"ahler classes $\alpha'_i \in H^{1,1}(X'_i, \R)$ such that $\rho_j^*\alpha'_i = \alpha_i$. For all $1 \leqslant j \leqslant m$, one has $H^{2,0}(Y_j) = 0$ and thus $Y_j$ is projective and $\beta_j$ is the class of a $\R$-divisor on $Y_j$. By \cite[3.9.1]{BCHM} we have a contraction $\nu_j : Y_j \to Y_j'$ together with K\"ahler classes $\beta'_j \in H^{1,1}(Y'_j, \R)$ such that $\nu_j^*\beta'_i = \beta_i$. By setting $\widehat{X}' := T \times \prod_{i = 1}^{n} X'_i \times \prod_{j = 1}^{m} Y'_j$ and $\widehat{\omega} := \mathrm{pr}_T^*\alpha_T + \sum_{i = 1}^{n} \mathrm{pr}_{X'_i}^*\alpha'_i + \sum_{j = 1}^{m} \mathrm{pr}_{Y'_j}^*\beta'_j$ we get a contraction $\widehat{\rho} := \Id_T \times \prod_{i = 1}^{n} \rho_i \times \prod_{j = 1}^{m} \nu_j : \widehat{X} \to \widehat{X}'$ such that $\widehat{\rho}^*\widehat{\omega} = \widehat{\alpha}$. Since $K_{\widehat{X}}$ is trivial and $\widehat{Y}$ is normal, $K_{\widehat{Y}}$ is trivial. In particular $\widehat{\rho}$ is crepant and $\widehat{Y}$ has canonical singularities. Since canonical singularities are rational (see \cite[Theorem 3.12]{FujinoMMPForProjMor} for a statement in the complex analytic setting) we can apply Lemma \ref{lem Galois} to conclude.
\end{proof}

\subsection{Proof of Proposition \ref{prop cas HK}}

Let $X$ be a \HK manifold of dimension $2n$. Recall that $H^2(X, \Z)$ comes endowed with a quadratic form $q_X$ called the \textit{Beauville-Bogomolov form}. This form satisfies  the \textit{Fujiki relation} $q_X(\beta, \beta)^n = c_X \int_X \beta^{2n}$ for a constant $c_X >0$. We will denote by $\mathrm{Pos}(X) \subset H^{1,1}(X, \R)$ the open cone formed by elements of positive $q_X$ square. Notice that by the Fujiki relation any big and nef class on $X$ lies inside $\mathrm{Pos}(X)$. The following result gives a clear description of the K\"ahler cone of $X$. 

\begin{thm}[{\cite[Theorem 1.2]{BoucksomConeHK}, \cite[Proposition 3.2]{HuybrechtsKahlerCone}}]\label{thm cone Kahler}
    One has 
    \[ \Kah(X) = \{ \beta \in \mathrm{Pos}(X) \mid \text{$\beta \cdot [C] > 0$ for all rational curves $C \subset X$}\}.\]
\end{thm}

Following \cite[Definition 4.12]{AmerikVerbitsky}, we will say that a class $\delta \in H_2(X, \Q)$ is \textit{minimal} if $\delta^\perp \subset H^{1,1}(X, \R)$ contains a wall of $\Nef(X)$. By Theorem \ref{thm cone Kahler} and Lemma \ref{lem cone}, a minimal class $H_2(X, \Q)$ is represented, up to multiplication by a scalar, by a rational curve. Minimal classes in $H_2(X, \Q)$ are particular instances of a more general notion introduced by Amerik and Verbitsky in \cite{AmerikVerbitsky} called \textit{MBM classes}. As part of the proof the cone conjecture for \HK manifold, Amerik and Verbitsky showed in \cite{AmerikVerbitskyannallesENS} that integral MBM classes have bounded square whenever $b_2(X) \geqslant 6$. As observed by Bakker and Lehn in \cite{BakkerLehnAGlobalTorelli}, this has the following consequence.

\begin{prop}[{\cite[Theorem 5.3]{AmerikVerbitskyannallesENS}, \cite[Remark 3.16]{BakkerLehnAGlobalTorelli}}]\label{prop faces}
    Assume that $b_2(X) \geqslant 6$. Then a point $\alpha \in \partial\Nef(X) \cap \mathrm{Pos}(X)$ lies on finitely many walls formed by extremal classes. 
\end{prop}

As a consequence, any class $\alpha \in H^{1,1}(X, \R)$ that is big and nef but not K\"ahler is contained in a face of $\Nef(X)$ that interesects $\mathrm{Pos}(X)$. Using Verbitsky's ergodicity theorem (\cite{VerbitskyErgodic}) together with Kawamata's basepoint-free theorem, Bakker and Lehn proved the following. 

\begin{thm}[{\cite[Corollary 5.9]{BakkerLehnAGlobalTorelli}, see also \cite[Theorem 4.6]{AmerikverbitskyContractionCenters}}]\label{thm contraction BL}
    Let $X$ be a \HK manifold and let $\tau \subset H^{1,1}(X)$ be a face of $\Nef(X)$ such that $\tau \cap \mathrm{Pos}(X) \neq \emptyset$. Assume that $\dim(\tau) \geqslant 3$. Then there exists a bimeromorphic contraction $\rho : X \to X'$ with the property that a curve $C \subset X$ is contracted under $\rho$ if and only if $[C] \in \tau^\perp \subset H_2(X, \Q)$. 
\end{thm}

\begin{proof}[Proof of Proposition \ref{prop cas HK}]
    We can assume that $\gamma$ is not K\"ahler and thus, by Theorem \ref{thm cone Kahler}, $N \neq 0$ and in particular $b_2(X) \geqslant 6$. By Proposition \ref{prop faces}, $\alpha$ is contained inside $\mathrm{Pos}(X)$ and lies on a face of $\Nef(X)$. Let $\tau \subset H^{1,1}(X, \R)$ be the minimal face containing $\alpha$ so that $\alpha \in \mathrm{Int}(\Nef(X) \cap \tau)$. We claim that $\tau = N^\perp$. Indeed, on the one hand we have that $\tau$ is an intersection of hyperplanes of the form $[C]^\perp$ for rational curves $C$ such that $[C] \cdot \alpha = 0$, and thus $N^\perp \subset \tau$. On the other hand, if $C$ is rational curve such that $[C]^\perp \cdot \alpha = 0$, then by Lemma \ref{lem cone} we have $\tau \subset [C]^\perp$, thus $\tau \subset N^\perp$. By our hypothesis, this implies that $\dim(\tau) = h^{1,1}(X) - \dim(N) \geqslant 3$. By Theorem \ref{thm contraction BL} there exists a contraction $\rho : X \to X'$ that contracts exactly the curves whose class lies inside $N$ and the result then follows from Lemma \ref{lem cone kahler contracté}.
\end{proof}

\begin{rmk}\label{rmk module}
    In \cite{BakkerLehnAGlobalTorelli} the authors prove a global Torelli theorem for primitive symplectic varieties admitting a crepant resolution. They obtain in particular that every component of the locally trivial marked moduli space is isomorphic to a refined period domain, the isomorphism being given by assigning to a marked pair $(X, \varphi)$ its period together with the chamber of a certain "K\"ahler type" chamber decomposition containing $\varphi(\Kah(X))$ (\cite[Proposition 5.21]{BakkerLehnAGlobalTorelli}). They however note that their proof doesn't imply that the chamber in question coincides with the image of the K\"ahler cone of the variety. The proof of Proposition \ref{prop cas HK} implies that this is indeed the case ; the chambers of the K\"ahler type decomposition do correspond to the K\"ahler cones of the parametrized varieties, making the situation completely analogous to that of \HK manifolds (\cite[Theorem 5.16]{MarkmanSurvey}).
\end{rmk}

\bibliographystyle{amsalpha}
\bibliography{biblio.bib}

\providecommand{\bysame}{\leavevmode\hbox to3em{\hrulefill}\thinspace}
\providecommand{\MR}{\relax\ifhmode\unskip\space\fi MR }
\providecommand{\MRhref}[2]{%
  \href{http://www.ams.org/mathscinet-getitem?mr=#1}{#2}
}
\providecommand{\href}[2]{#2}
\begin{thebibliography}{BCHM10}

\bibitem[AV15]{AmerikVerbitsky}
Ekaterina Amerik and Misha Verbitsky, \emph{{Rational Curves on Hyperkähler Manifolds}}, International Mathematics Research Notices \textbf{2015} (2015), no.~23, 13009--13045.

\bibitem[AV17]{AmerikVerbitskyannallesENS}
\bysame, \emph{Morrison-{Kawamata} cone conjecture for hyperk{\"a}hler manifolds}, Ann. Sci. {\'E}c. Norm. Sup{\'e}r. (4) \textbf{50} (2017), no.~4, 973--993.

\bibitem[AV21]{AmerikverbitskyContractionCenters}
\bysame, \emph{Contraction centers in families of hyperk{\"a}hler manifolds}, Sel. Math., New Ser. \textbf{27} (2021), no.~4, 26.

\bibitem[BCHM10]{BCHM}
Caucher Birkar, Paolo Cascini, Christopher~D. Hacon, and James McKernan, \emph{Existence of minimal models for varieties of log general type}, J. Am. Math. Soc. \textbf{23} (2010), no.~2, 405--468.

\bibitem[Bea83]{BeauvillePremièreClasse}
Arnaud Beauville, \emph{Vari{\'e}t{\'e}s k{\"a}hleriennes dont la premi{\`e}re classe de {Chern} est nulle}, J. Differ. Geom. \textbf{18} (1983), 755--782 (French).

\bibitem[BG13]{BoucksomGuedjRegularizing}
S{\'e}bastien Boucksom and Vincent Guedj, \emph{Regularizing properties of the {K{\"a}hler}-{Ricci} flow}, An introduction to the K\"ahler-Ricci flow. Selected papers based on the presentations at several meetings of the ANR project MACK, Cham: Springer, 2013, pp.~189--237.

\bibitem[BL21]{BakkerLehnAGlobalTorelli}
Benjamin Bakker and Christian Lehn, \emph{A global {Torelli} theorem for singular symplectic varieties}, J. Eur. Math. Soc. (JEMS) \textbf{23} (2021), no.~3, 949--994.

\bibitem[Bou01]{BoucksomConeHK}
S{\'e}bastien Boucksom, \emph{The {K{\"a}hler} cone of a hyperk{\"a}hler manifold}, C. R. Acad. Sci., Paris, S{\'e}r. I, Math. \textbf{333} (2001), no.~10, 935--938 (French).

\bibitem[Car57]{CartanQuotientEspace}
Henri Cartan, \emph{Quotient d'un espace analytique par un groupe d'automorphismes}, Princeton Math. Ser. \textbf{12} (1957), 90--102 (French).

\bibitem[Dem85]{DemaillyMesureDeMA}
Jean-Pierre Demailly, \emph{M{\'e}sures de {Monge}-{Amp{\`e}re} et caract{\'e}risation g{\'e}om{\'e}trique des vari{\'e}t{\'e}s alg{\'e}briques affines}, M{\'e}m. Soc. Math. Fr., Nouv. S{\'e}r., vol.~19, Soci{\'e}t{\'e} Math{\'e}matique de France (SMF), Paris, 1985.

\bibitem[DH24]{DasHaconTranscendentalMMPforProj}
Omprokash Das and Christopher Hacon, \emph{Transcendental {Minimal} {Model} {Program} for {Projective} {Varieties}}, Preprint, {arXiv}:2412.07650 [math.{AG}] (2024), 2024.

\bibitem[DHP24]{DasHaconPaun4-dimensional}
Omprokash Das, Christopher Hacon, and Mihai P{\u{a}}un, \emph{On the 4-dimensional minimal model program for {K{\"a}hler} varieties}, Adv. Math. \textbf{443} (2024), 68, Id/No 109615.

\bibitem[FT18]{FilipTosatti}
Simion Filip and Valentino Tosatti, \emph{Smooth and {Rough} {Positive} {Currents}}, Annales de l'Institut Fourier \textbf{68} (2018), no.~7, 2981--2999.

\bibitem[Fuj22]{FujinoMMPForProjMor}
Osamu Fujino, \emph{Minimal model program for projective morphisms between complex analytic spaces}, Preprint, {arXiv}:2201.11315 [math.{AG}] (2022), 2022.

\bibitem[GH78]{GriffithsHarrisPrinciples}
Phillip Griffiths and Joseph Harris, \emph{Principles of algebraic geometry}, Pure and {Applied} {Mathematics}. {A} {Wiley}-{Interscience} {Publication}. {New} {York} etc.: {John} {Wiley} \& {Sons}. {XII}, 813 p., 1978.

\bibitem[GK20]{GrafKirschnerFiniteQuotients}
Patrick Graf and Tim Kirschner, \emph{Finite quotients of three-dimensional complex tori}, Ann. Inst. Fourier \textbf{70} (2020), no.~2, 881--914 (English).

\bibitem[GR84]{GrauertRemmertCohAnSheaves}
Hans Grauert and Reinhold Remmert, \emph{Coherent analytic sheaves}, Grundlehren Math. Wiss., vol. 265, Springer, Cham, 1984.

\bibitem[H{\"o}r21]{HoringAdjoint}
A.~H{\"o}ring, \emph{Adjoint {{\((1,1)\)}}-classes on threefolds}, Izv. Math. \textbf{85} (2021), no.~4, 823--830.

\bibitem[HP16]{HoringPeternellMMP3fold}
Andreas H{\"o}ring and Thomas Peternell, \emph{Minimal models for {K{\"a}hler} threefolds}, Invent. Math. \textbf{203} (2016), no.~1, 217--264.

\bibitem[Huy99]{HuybrechtsBasicResults}
Daniel Huybrechts, \emph{Compact hyperk{\"a}hler manifolds: basic results}, Inventiones mathematicae \textbf{135} (1999), no.~1, 63--113.

\bibitem[Huy03]{HuybrechtsKahlerCone}
\bysame, \emph{The {K{\"a}hler} cone of a compact hyperk{\"a}hler manifold}, Math. Ann. \textbf{326} (2003), no.~3, 499--513.

\bibitem[Mar11]{MarkmanSurvey}
Eyal Markman, \emph{{A survey of Torelli and monodromy results for holomorphic-symplectic varieties}}, Complex and Differential Geometry, Springer Berlin Heidelberg, 2011, pp.~257--322.

\bibitem[Tos09]{TosattiLimitsOfCY}
Valentino Tosatti, \emph{Limits of {Calabi}-{Yau} metrics when the {K{\"a}hler} class degenerates}, J. Eur. Math. Soc. (JEMS) \textbf{11} (2009), no.~4, 755--776.

\bibitem[Tos25]{tosatti2025ricciflatmetricscalabiyaumanifolds}
Valentino Tosatti, \emph{Ricci-flat metrics on {Calabi-Yau} manifolds}, 2025.

\bibitem[Var84]{VarouchasInvent}
J.~Varouchas, \emph{Stabilit{\'e} de la classe des vari{\'e}t{\'e}s {Kaehleriennes} par certains morphismes propres}, Invent. Math. \textbf{77} (1984), 117--127 (French).

\bibitem[Var89]{VarouchasMathAnn}
Jean Varouchas, \emph{K{\"a}hler spaces and proper open morphisms}, Math. Ann. \textbf{283} (1989), no.~1, 13--52.

\bibitem[Ver15]{VerbitskyErgodic}
Misha Verbitsky, \emph{Ergodic complex structures on hyperk{\"a}hler manifolds}, Acta Math. \textbf{215} (2015), no.~1, 161--182.

\bibitem[Ver17]{VerbitskyErgodicErratum}
Misha Verbitsky, \emph{Ergodic complex structures on hyperkahler manifolds: an erratum}, Preprint, {arXiv}:1708.05802 [math.{AG}] (2017), 2017.

\end{thebibliography}

\end{document}